\documentclass{amsart}

\usepackage[utf8]{inputenc}
\usepackage[T1]{fontenc}
\usepackage[english]{babel}
\usepackage{color}
\usepackage{amsmath} 
\usepackage{amssymb} 
\usepackage{amsthm}
\usepackage{graphicx}
\usepackage[colorlinks=true, linkcolor=blue]{hyperref}
\usepackage{cancel}

\usepackage{lmodern}
\usepackage{microtype}

\usepackage{enumerate}
\usepackage[shortlabels]{enumitem}

\theoremstyle{plain}
\newtheorem{theorem}{Theorem}[section]
\newtheorem*{theorem*}{Theorem}
\newtheorem{proposition}[theorem]{Proposition}
\newtheorem{lemma}[theorem]{Lemma}

\newtheorem*{cor*}{Corollary}

\newtheorem{defi}[theorem]{Definition}

\newcommand {\R} {\mathbb{R}} \newcommand {\Z} {\mathbb{Z}}
\newcommand {\T} {\mathbb{T}} \newcommand {\N} {\mathbb{N}}

\newcommand {\p} {\partial}
\newcommand {\dt} {\partial_t}

\newcommand{\cutoff}[1]{\left\ulcorner{#1}\right\urcorner}

\begin{document}
\title{On nonlinear Landau damping and Gevrey regularity}
\begin{abstract}
  In this article we study the problem of nonlinear Landau damping for the Vlasov-Poisson
  equations on the torus. As our main result we show that for perturbations
  initially of size $\epsilon>0$ and time intervals $(0,\epsilon^{-N})$ one
  obtains nonlinear stability in regularity classes larger than Gevrey $3$, uniformly in $\epsilon$.
   As a complementary result we construct families of Sobolev regular initial
   data which exhibit nonlinear Landau damping.
   Our proof is based on the methods of Grenier, Nguyen and Rodnianski \cite{grenier2020landau}.
\end{abstract}
\author{Christian Zillinger}
\address{Karlsruhe Institute of Technology, Englerstraße 2,
  76131 Karlsruhe, Germany}
\email{christian.zillinger@kit.edu}
\maketitle
\tableofcontents

\section{Introduction}
\label{sec:intro}

In this article we consider the nonlinear stability problem for the
Vlasov-Poisson equations
\begin{align}
  \label{eq:VP}
  \begin{split}
  \dt f + v\cdot \nabla_x f + F \cdot \nabla_v f &=0, \\
  (t,x,v) &\in \R \times \T^d \times \R^d, \\
  \rho(t,x)&= \int f(t,x,v) dv, \\
  F(t,x) &= \nabla \Delta^{-1} \rho(t,x), \\
  (t,x,v) &\in (0,\epsilon^{-N})\times \T \times \R,
  \end{split}
\end{align}
for \emph{finite}, but very large times and for perturbations initially of size
$\epsilon>0$.
For simplicity of presentation we restrict to perturbations around the zero solution $f(t,x,v)= 0$.
Here $f(t,x,v)\in \R$ models the phase space-density of a plasma and $F(t,x) \in
\R$
corresponds to a mean-field electric force field generated by the spatial
density $\rho(t,x)\in \R$.

The study of the long-time behavior of the Vlasov-Poisson equations and, in particular, the
phenomenon of Landau damping (decay of $F(t,x)$ as $t\rightarrow \infty$ at very
fast rates), is a very active field of research. For an overview we refer to the
seminal works of Villani and Mouhot \cite{Villani_long,Villani_script,bedrossian2016landau} who for
the first time established nonlinear stability and Landau damping.

Due to nonlinear resonances \cite{malmberg1968plasma} these results require
extremely strong regularity assumptions. Indeed in
\cite{bedrossian2016nonlinear} (see also \cite{zillinger2020landau} and \cite{grenier2022plasma}) it is shown that the
nonlinear equations exhibit chains of resonances and associated growth by
\begin{align*}
  \exp(|\epsilon\eta|^{1/3})
\end{align*}
for perturbations frequency localized at $\eta$ (with respect to $v$).
Hence no uniform (in time) stability results can be expected to hold in weaker than Gevrey $3$
regularity (that is, $L^2$ spaces with such an exponential decay in Fourier space).
Nevertheless, as shown in \cite{zillinger2020landau} in principle the ``physical
notion'' of Landau damping (that is, $F(t,x)$ decays in time) does not require
stability of $f(t,x,v)$. In particular, that notion of damping might be more ``robust''
than suggested by the high regularity requirements and suggests that Landau
damping might persist even at lower than Gevrey regularity of perturbations.

As a first step towards such a result in this article we adapt the
methods of \cite{grenier2020landau} and ask the following questions:
\begin{itemize}
\item If we consider perturbations initially of size $0< \epsilon \ll 1$, how
  does the possible (optimal) norm inflation depend on $\epsilon$?
\item Can the methods of \cite{grenier2020landau} be modified to reach optimal
  Gevrey classes?
\item If we consider a finite time interval
  \begin{align*}
  (0,\epsilon^{-N}), 
  \end{align*}
  how does this change upper and lower bounds and can we establish stability in
  better than Gevrey $3$ regularity uniformly in $\epsilon$?
\end{itemize}
This coupling between the size of the perturbation and the time scale is
motivated by works on echo chains in the inviscid Boussinesq equations
\cite{bedrossian21,zillinger2022stability}, where one naturally is restricted to
a time scale $(0,\epsilon^{-2})$. 

Our main results are summarized in the following theorem.
\begin{theorem}
  \label{thm:main}
  Let $0<\epsilon \ll 0.1$ and $T=\epsilon^{-N}$.
  Then there exists $\gamma=\gamma(N)$ independent of $\epsilon$ such that the
  nonlinear Vlasov-Poisson equations are stable in Gevrey $\frac{1}{\gamma}>3$.

  More precisely, there exists $\beta\geq \frac{1}{12}$ and $C>0$ such that if the Fourier transform of initial
  data satisfies
  \begin{align}
    \label{eq:E1initial}
    & \sum_k \int \langle k,\eta \rangle^{8} \exp(C \log(T) \min (\epsilon^{\beta}\langle k,\eta \rangle^{1/3}, \langle k,\eta \rangle^{\gamma}))\\
    &\quad (|\mathcal{F}|f_0(k,\eta)|^2 + |\p_\eta\mathcal{F}f_0(k,\eta)|^2) d\eta \leq \frac{1}{100} \epsilon, 
  \end{align}
  then the bound \eqref{eq:E1initial} remains true for $g(t,x,v)=f(t,x-tv,v)$ for all times $t \in (0,T)$
  up to a loss in the constants $C$ and $\frac{1}{100}$.
    Moreover, on that same time interval the force field perturbation satisfies
    the bound
    \begin{align*}
      |F(t,x)|\lesssim \epsilon \exp(-\epsilon^{\beta}t^{\gamma}).
    \end{align*}
\end{theorem}
We in particular emphasize the following differences and improvements compared
to \cite{grenier2020landau}:
\begin{itemize}
\item In \cite{grenier2020landau}, whose methods we adapt, a similar result is
  established with
  \begin{align*}
    \exp(C \langle k,\eta \rangle^{1/3+ \delta}); \delta>0,
  \end{align*}
  instead and $T$ is allowed to be infinite.
\item The present result reaches $\frac{1}{3}$ ($\delta=0$) at the cost of
  time-dependent prefactor $\log(T)$. However, for $T=\epsilon^{-N}$ this can be
  absorbed into a slight loss of $\beta$ (and $\gamma$).
\item As our main novelties we highlight the frequency cut-off and the improved
  Gevrey classes. To the author's knowledge this is the first nonlinear Landau
  damping result (for finite time) for \emph{generic small data} in sub Gevrey $3$
  regularity.
\item As we discuss in the following Lemma \ref{lem:trivial}, it is easy
  to construct special data at arbitrarily low regularity which exhibits Landau
  damping as $t\rightarrow \infty$. However, this data needs to satisfy rather
  restrictive Fourier support assumptions and is unstable as $t\rightarrow
  -\infty$.
  As seen from the norm inflation results of
  \cite{bedrossian2016nonlinear,zillinger2020landau} for generic data some level
  of Gevrey regularity is necessary for uniform in time stability.
\item As we discuss in Section \ref{sec:model}, a model for plasma echoes suggests that optimal
  bounds should be given by $\beta=\frac{1}{3}$ and $\gamma_N=
  \frac{1}{3}\frac{3N-2}{3N-1}$. However, since the current method of proof not
  only requires solutions to remain bounded but to decay with an integrable rate
  $(1+|t|)^{-\sigma+2}$, this restricts us to slightly smaller values of $\beta$ and
  hence larger values of $\gamma= \frac{1}{3} - \frac{\beta}{2N}$.
\end{itemize}

As an independent result, the following lemma constructs examples of ``trivial'' solutions, which
are only Sobolev regular but nevertheless exhibit Landau damping (see also \cite{grenier2022plasma}).
\begin{lemma}
  \label{lem:trivial}
   Let $\psi \in \mathcal{S}(\R)$ be a Schwartz function whose Fourier transform
   is supported in a ball of radius $0.1$ around $0$.
   Then for any $s\geq 0$ and any sequence $(c_k)_{k} \in \ell^2(\N)$ the function
   \begin{align}
     \label{eq:trivialexample}
     f(t,x,v)= \Re \sum_{k\geq 1}c_k (1+k^2)^{\frac{s}{2}}  e^{ikx} e^{i(-k-kt)v} \psi(v)
   \end{align}
   is an element of the Sobolev space $H^s(\T\times\R)$ for all times $t>0$ and is a solution of
   the Vlasov-Poisson equations for $t \in (0,\infty)$. Moreover,
   $f(t,x+tv,v)=f(0,x,v)$ is independent of time and $F(t,x)=0$. Hence this
   solution trivially exhibits both scattering and Landau damping.
 \end{lemma}

\begin{proof}[Proof of Lemma \ref{lem:trivial}]
  We note that by the assumption on the Fourier support of $\psi$, the Fourier
  transform of the functions
  \begin{align*}
    (1+k^2)^{\frac{s}{2}} e^{ikx} e^{i(-k-kt)v} \psi(v)
  \end{align*}
  is supported in a ball of size $0.1$ around the frequency $(k,-k-kt)$.
  Since $k\geq 1$ and $t>0$, it follows that these supports are disjoint for
  different values of $k$ and hence for $\psi\neq 0$ the series is convergent in $H^s$ if and
  only if $(c_k)_k \in \ell^2(\N)$.
  Furthermore, since $k+kt \geq 1 > 0.1$ it follows that the spatial density
  $\rho(t,x)$ satisfies
  \begin{align*}
    \rho(t,x)&=\int f(t,x+tv,v) dv = \int \Re \sum_{k\geq 1} c_k e^{ikx} e^{i\eta_k v} \psi(v)\\
    & = \Re \sum_{k\geq 1} e^{ikx} \mathcal{F}(\psi)(\eta_k) =0
  \end{align*}
  and hence also $F(t,x)=\nabla \Delta^{-1}\rho=0$ is trivial. Thus the Vlasov-Poisson equations reduce to
  the free transport equations, which $f(t,x,v)$ solves by construction.
\end{proof}
Such ``traveling wave'' solutions also form the core of nonlinear instability results
\cite{zillinger2020landau,bedrossian2016nonlinear}, where small high-frequency
perturbations of these waves exhibit norm inflation.
More generally one can consider a factor $e^{i(-\eta_k-kt)v}$
with $\eta_k>0.1$ positive and increasing in $k$.
However, we emphasize that the behavior changes drastically when this sign
condition is allowed to be violated or if we require Landau damping for both
$t\rightarrow \infty$ and $t\rightarrow - \infty$. In these cases one needs to
require $c_k$ to decay sufficiently rapidly to satisfy the assumptions of
Theorem \ref{thm:main}.

The remainder of this article is structured as follows:
\begin{itemize}
\item In Section \ref{sec:model} we briefly discuss the plasma echo mechanism
  in terms of a toy model, which exhibits exactly the growth expressed in Theorem
  \ref{thm:main}.
\item In Section \ref{sec:generatorfunctions} we recall the generator function method of \cite{grenier2020landau}.
  Here the added $\epsilon$ dependence and the cut-off constitute the main new
  effects and require more precise estimates. Furthermore, we restructure the
  proof to highlight the role of traveling waves and establish improved
  bounds above a frequency cut-off.
\item The estimates use a bootstrap approach. Here the most important step is
  given by Subsection \ref{sec:E2} establishing improved control of $\rho$
  assuming control of $f$.
\item Conversely, Subsection \ref{sec:E1} establishes improved control of $f$
  given control of $\rho$. Both conditional results are then combined to
  prove Theorem \ref{thm:main}.
\end{itemize}

\section{Plasma echoes and two heuristic models}
\label{sec:model}

In this section we briefly discuss the main norm inflation mechanism of the
Vlasov-Poisson equations, known as \emph{plasma echoes} and the effects of a
time cut-off.
These resonances are also experimentally observed \cite{malmberg1968plasma}.
The interested reader is referred to the seminal works of Villani, Mouhot and
Bedrossian \cite{Villani_script,bedrossian2016nonlinear,zillinger2020landau} for
a more detailed discussion.

As a heuristic model, let $\eta \in \R$ and $k \in \Z$ be given and let $\psi \in
\mathcal{S}(\R)$ be a Schwartz function as in Lemma \ref{lem:trivial}. Then by a
similar argument as in Lemma \ref{lem:trivial} the function
\begin{align*}
  f(t,x,v) = \epsilon \cos(x -tv) \psi(v) + \epsilon \sin(kx + (\eta-kt)v) \psi(v)
\end{align*}
is a solution of nonlinear Vlasov-Poisson equations \eqref{eq:VP} for all times $t$ such
that $|t|>0.1$ and $|\eta-kt| >0.1$. We refer to both summands as (traveling) waves.
According to the linearized dynamics around $0$ (that is, the free transport
dynamics), both waves do not interact and exhibit weak convergence in $L^2$
as $t\rightarrow \infty$.
However, when considering the full nonlinear problem with the same initial data at
time $t=0.1$, the nonlinearity $F\cdot \nabla_v f$ introduces a correction around the time
$t:|\eta-kt|<0.1$.

For our model problems we insert different waves for $f$ and $F[\rho]$,
which corresponds to considering parts of the first Duhamel iteration.
We thus obtain a correction involving a time integral of
\begin{align}
  \label{eq:Duhamelpart1}
 F[\int \epsilon \sin(kx + (\eta-kt)V) \psi(V)dV ] \cdot \nabla_v  \epsilon \cos(x-tv) \psi(v)
\end{align}
and
\begin{align*}
  F[\int \epsilon \cos(x-tV) \psi(V) dV] \nabla_v \epsilon \sin(kx+(\eta-kt)v) \psi(v),
\end{align*}
respectively. In the naming of \cite{bedrossian2016landau, Villani_script} these
are model problems for the ``reaction'' and ``transport'' terms.

We begin by considering the model associated to \eqref{eq:Duhamelpart1}.
Changing to coordinates $(x-tv,v)$, taking a Fourier transform and inserting the
choice of $F=\nabla \Delta^{-1} \rho$, one deduces (see \cite{Villani_script})
that the correction is estimated to be of the size
\begin{align*}
  \epsilon^2 \frac{|\eta|}{|k|^3}\|\mathcal{F}\psi\|_{L^1},
\end{align*}
is localized to the frequency $k\pm 1$ in $x$ and $\eta$ in $v$ and occurs at around the time
$\frac{\eta}{k}$.
This time-localized, large correction is the physically observed \emph{echo}.

Furthermore, in principle it could happen that this correction in turn results in
another correction at the later time $\frac{\eta}{k-1}$, then at the time $\frac{\eta}{k-2}$ and so on.
One thus obtains the upper estimate of the toy model of \cite{Villani_long} by
\begin{align*}
  \prod_{l=1}^k  \epsilon \frac{|\eta|}{|l|^3},
\end{align*}
which is maximized for $k \approx \sqrt[3]{\epsilon|\eta|}$ and suggests an
upper bound on the norm inflation by
\begin{align}
  \label{eq:growthfactor}
  \sup_{k} \prod_{l=1}^k  \epsilon \frac{|\eta|}{|l|^3} \approx \exp(\sqrt[3]{\epsilon |\eta|}).
\end{align}
While this toy model only suggests an upper bound for stability results, in
\cite{bedrossian2016nonlinear} Bedrossian showed that such resonance chain
can indeed occur in the nonlinear problem (with certain constraints on
$\epsilon$ and $\eta$). Moreover, in \cite{zillinger2020landau} for a related model
(linearizing around the \emph{wave solution} $\epsilon \cos(x e_j -tv)\psi(v)$) we showed
that infinitely many such resonance chains can lead to norm blow-up as
$t\rightarrow \infty$ in any regularity class below Gevrey $3$.

However, we emphasize that for resonance to provide a large contribution, two competing conditions have to be satisfied (see
also \cite{zillinger2022stability}):
\begin{itemize}
\item One the one hand, the frequency $l$ should satisfy an \emph{upper bound}
  such that $\epsilon \frac{|\eta|}{|l|^3} \gtrsim 1$ to actually yield a large
  correction.
\item On the other hand, the frequency $l$ needs to satisfy a \emph{lower bound}
  so that the resonant time $\frac{\eta}{l}$ occurs before the time cut-off $T$.
\end{itemize}
Thus, when considering a finite time interval $(0,T)$ the heuristic estimate
\eqref{eq:growthfactor} should be modified to
\begin{align*}
  \prod_{\frac{\eta}{T} \leq l \leq \sqrt[3]{\epsilon \eta}}  \epsilon \frac{|\eta|}{|l|^3} 
  \lesssim 
  \begin{cases}
    \exp(\sqrt[3]{\epsilon |\eta|}) & \text{ if } \frac{|\eta|}{\sqrt[3]{\epsilon |\eta|}} \leq T, \\
    1 & \text{ else}.
  \end{cases}
\end{align*}
The transition between these cases occurs at the cut-off frequency
$\eta_*=\sqrt{\epsilon T^3}$.

In particular, for this toy model and $T=\epsilon^{-N}$ we thus obtain that for
all $|\eta|\leq \eta_*= \epsilon^{-3N+1}$,
\begin{align*}
  \epsilon |\eta| \leq |\eta|^{\frac{3N-2}{3N-1}}.
\end{align*}
The norm inflation in this model can thus be estimated by 
\begin{align*}
   \exp(|\min(\eta,\eta_*)|^{\gamma_N})
\end{align*}
with
\begin{align*}
  \gamma_N &= \frac{1}{3} \frac{3N-2}{3N-1},\\
  \eta_*&= \epsilon^{- \frac{3N-1}{2}}.
\end{align*}
A major aim of Theorem \ref{thm:main} is to establish that such a cut-off effect
holds also for the full nonlinear problem. However, in view of technical
challenges we here allow for different cut-offs and different powers of
$\epsilon$.
As we discuss in Section \ref{sec:E2} and Section \ref{sec:E1} this does not
seem to be just a technical issue. More precisely, while for the frequency
regimes corresponding to this first model it seems possible to reach exactly
these powers, other frequency regimes pose greater challenges, which we
illustrate in the following model.\\

In our second model we consider the contribution by
\begin{align}
  \label{eqLDuhamelpart2}
  \begin{split}
  & \quad F[\int \epsilon \cos(x-tV) \psi(V) dV] \nabla_v \epsilon \sin(kx+(\eta-kt)v) \psi(v) \\
  &\approx \epsilon \tilde{\psi}(t) \cos(x) \epsilon (\eta-kt)  \cos(kx+(\eta-kt)v) \psi(v),
\end{split}
\end{align}
where $\tilde{\psi}$ denotes the Fourier transform of $\psi$ with respect to $v$.
Hence for frequencies $\eta$ much larger than $kt$, this contribution suggests
that $g(t,x,v):=f(t,x-tv,v)$ should behave as a solution of
\begin{align*}
  \dt g \approx \epsilon \tilde{\psi}(t) \cos(x-tv) \p_v g.
\end{align*}
We stress that this is not an effect one would see for $\rho=\int f dv$, since
there $\eta \approx kt$ due to the velocity integral.

Even assuming that $\psi$ is very smooth and hence that $\tilde{\psi}(t)$ is
decaying rapidly, this transport type equation poses great challenges for
estimates since $\p_v$ is an unbounded operator.
In particular, while for analytic regularity this contribution could be easily ``hidden'' in a
loss of constant (that is, consider a weight $\exp(z(t)\langle \eta \rangle)$
with $-\dt z\gg \tilde{\psi}(t)$) any weaker Gevrey class will have to account
for the fact that $z$-derivative a priori only gains fractional regularity in
$v$.
Hence, in Section \ref{sec:E1} we need to exploit that $\p_v$ is an
anti-symmetric operator on $L^2$ and that corresponding commutators in our $L^2$
based Gevrey spaces provide a ``gain'' of one derivative.

The requirements highlighted by both of these models and the fact that the
method of proof not only requires that $\rho$ remains bounded, but rather
\begin{align*}
  \|\rho(t,x)\| \leq \epsilon (1+|t|)^{-\sigma+1}
\end{align*}
for a given power $\sigma>3$ determines our cut-off and $\epsilon$ dependences
in Theorem \ref{thm:main}.

In the remainder of the article we to adapt the method of proof of
\cite{grenier2020landau} to establish non-linear stability estimates and
incorporate these effects.

\section{Generator functions and cut-offs}
\label{sec:generatorfunctions}
In our proof we follow the method of \cite{grenier2020landau} with (major)
modifications to account for the frequency cut-off and the $\epsilon$
dependence.
We briefly discuss the overall strategy of the proof and state the main
estimates as propositions, which we use to establish Theorem
\ref{thm:main}. The proofs of these propositions is given in Sections
\ref{sec:E2} and \ref{sec:E1}.

Considering the structure of the Vlasov-Poisson equations \eqref{eq:VP}, we study solutions
in coordinates moving with free transport and denote
\begin{align*}
  g(t,x,v) = f(t,x+tv,v).
\end{align*}
Then the Vlasov-Poisson equations can be equivalently expressed as a coupled
system for $\rho$ and $g$:
\begin{align}
  \label{eq:VPsys}
  \begin{split}
    \dt g &=- F[\rho](t,x+tv)\cdot (\nabla_v + t \nabla_x)g, \\
    \rho &= \int g(t,x-tv,v) dv, \\
    F&= \nabla W *_x \rho, \\
    \dt \rho &=- \int F(t,x) (\nabla_v + t \nabla_x)g(t,x-tv,v) dv.
  \end{split}
\end{align}
Similarly to usual Cauchy-Kowaleskaya approaches using time-dependent Fourier
multipliers \cite{bedrossian2016landau}, one further introduces two parameter dependent energy functionals,
where at a later stage the parameter will also be chosen depending on time.
\begin{defi}[Generator functions (c.f. \cite{grenier2020landau})]
  \label{defi:gen}
  Let $0<\epsilon\ll 1$, $C>0$, $\sigma > 3$, $\alpha \in (\frac{1}{3}, \frac{1}{2})$ and $\eta_{*}\gg 1$ be given constants.
  We further introduce the short-hand-notation
  \begin{align*}
    \cutoff{k,\eta}=
    \begin{cases} 
      \epsilon^\beta (1+ k^2+\eta^2)^{\frac{1}{6}}& \text{ if } |k|+|\eta| \leq \eta_*, \\
      \epsilon^{\beta'}(1+ k^2+\eta^2)^{\frac{\gamma}{2}} & \text{ if } |k|+ |\eta| \geq 2 \eta_*,
    \end{cases}
  \end{align*}
  with a smooth interpolation in the remaining region.
  We refer to these cases as \emph{below the cut-off} and \emph{above the cut-off}, respectively.
  Here the constant are chosen under the constraints 
  \begin{align}
    \label{eq:betaconstraints}
    \begin{split}
    0 \leq \beta &\leq \frac{1}{3\sigma}, \\
    0 \leq \beta' &\leq \frac{\gamma}{\sigma}, \\
      \frac{1-\alpha}{2} \leq \gamma &\leq \frac{1}{3}, \\
      \epsilon^{\beta'-\beta} &= \eta_*^{\frac{1}{3}-\gamma}, \\
      \eta_*&\geq T^2.
  \end{split}
  \end{align}
  Then for given functions $g$ and $\rho$ and a given parameter $\gamma>0$, for
  any $z\geq 0$ we define the (possibly infinite) energies
  \begin{align}
    \label{eq:defE1}
  E_1(z) = \|\exp(C z \cutoff{k,\eta}^{1/3}) \langle k,\eta\rangle^\sigma (\tilde{g}, \p_\eta \tilde{g})\|_{L^2(\Z\times \R)}^2
\end{align}
and
\begin{align}
  \label{eq:defE2}
  E_{2}(z) = \||k|^{-\alpha} \exp(C z \cutoff{k,kt}^{1/3}) \langle k,kt \rangle^\sigma \tilde{\rho}\|_{l^\infty(\Z)},
\end{align}
where $\tilde{g}(k,\eta)$ and $\tilde{\rho}(k)$ denote the respective Fourier
transforms.
\end{defi}
For the proof of Theorem \ref{thm:main} with
$T=\epsilon^{-N}$ our constants are chosen as
\begin{align*}
  \eta_* &= \epsilon^{-2N}, \\
  \beta &=\frac{1}{3\sigma}, \\
  \beta' &=0, \\
  \gamma &= \frac{1}{3} - \frac{\beta}{2N}. 
\end{align*}
We thus obtain stability in the Gevrey class $\frac{1}{\gamma}>3$ uniformly in
$0<\epsilon\ll 1$.

We remark that in \cite{grenier2020landau} stability is established with an
exponential weight
\begin{align*}
  \exp(z\langle k,\eta \rangle^{1/3+\delta})
\end{align*}
for $\delta>0$.  These new generator functions introduce an improved exponent
$1/3$, the gain of a factor $\epsilon^{\beta}$ and, most importantly, an
improved sub $\frac{1}{3}$ growth past a cut-off.

Our main aim in the following is to show that for the above choices of constants
one may find $z(t)\geq 1$ such that, if initially
\begin{align*}
  \sqrt{E_1} + E_2 \lesssim \epsilon,
\end{align*}
then this estimate remains valid for all
times smaller than $T$.

More precisely, we claim that $E_1$ satisfies the following estimate.
\begin{proposition}
  \label{prop:E1}
  Let $g,\rho$ be a solution of \eqref{eq:VPsys} and let $\eta_*$ and $\beta$ be
  as in Theorem \ref{thm:main}. Let further $E_1(z), E_2(z)$ denote the now
  time-dependent generator functions as in Definition \ref{defi:gen}. Then there
  exist universal constants $C_1,C_2>0$ such that for
  all times $0<t<T$ it holds that
  \begin{align*}
    \dt E_1(z) \leq C_1 E_2(z) E_1(z) + C_2  \epsilon^{-\beta} C^{-1} (1+t) E_2(z) \p_z E_1(z). 
  \end{align*}
\end{proposition}
This result is analogous to \cite[Proposition 4.1]{grenier2020landau} but
considers our modified generator functions and hence include the improved
$\epsilon$ dependence and frequency cut-off.
The proof of this proposition is given in Section \ref{sec:E2}.
Assuming these results for the moment, we note that if we can show that
\begin{align}
  \label{eq:assumpE2}
   C_2 \epsilon^{-\beta} C^{-1} (1+t) E_2(z) \leq \frac{1}{1+t}, 
\end{align}
then choosing
\begin{align*}
  z(t)= 2 \log(T) -\log(t),
\end{align*}
it follows that
\begin{align}
  \label{eq:IntIneq}
  \begin{split}
  \frac{d}{dt} E_1(z(t)) &\leq  C_1 E_2(z(t)) E_1(z(t)) \\
  \leadsto E_1(z(t)) \leq E_1(z(t))|_{t=0} \exp(\int_0^t E_2 ds) &\lesssim E_1(z(t))|_{t=0}.
\end{split}
\end{align}
Thus, the main challenge to the proof of Theorem \ref{thm:main} is given by
establishing a suitable decay bound on $E_2(z(t))$ with this choice of $z(t)$.
\begin{proposition}
  \label{prop:E2}
Let $E_1,E_2$ be as in Proposition \ref{prop:E1}. Furthermore, choose $z$ to be
time-dependent as
\begin{align*}
  z(t)= 2 \log(T) -\log(t)
\end{align*}
and suppose that on a time interval $[0,t_{*}]\subset [0.T]$ it holds that 
\begin{align}
  \label{eq:assumpE1}
  E_1(z(t))\leq 16 \epsilon^2.
\end{align}
Then on that same time interval we have the estimate 
\begin{align*}
  \begin{split}
  (1+t)^{\sigma-1} E_2(z(t)) &\leq \epsilon (1+t)^{\sigma-1} \\
  & \quad + c \sup_{0\leq s \leq t} (1+s)^{\sigma-1} E_2(z(s)),
\end{split}
\end{align*}
where $c>0$ only depends on the constant $C$ in Definition \ref{defi:gen} and
$c<1/2$ if $C$ is sufficiently large.
\end{proposition}

This estimate is similar to \cite[Lemma 4.4]{grenier2020landau} with the
following key differences:
\begin{itemize} 
\item The frequency transition $\cutoff{\cdot, \cdot}$ and the associated cut-off constitute a main new effect.
\item We here reach $1/3$ in the exponent. However, in turn the estimates only
  remain valid while $z(t)$ remains bounded below, which requires restricting to
  times smaller than $T$.
\item An analogous result can be established for $T\leq \infty$ with $\dt
  z(t)\approx (1+t)^{-1-\delta}$ and exponents $\cutoff{k,\eta}^{1/3+\delta}$
  (and smaller $\beta$) instead.  
\item In comparison to \cite{grenier2020landau}, the additional $\epsilon$
  dependence of $E_1,E_2$ implies a loss of powers of $\epsilon$ in some
  estimates. In particular, choosing $\beta, \beta'$ maximally under the
  constraints \eqref{eq:betaconstraints}, choosing $\epsilon$ small does not
  yield any further improvements. Hence, we need to carefully estimate all terms
  and establish a bound by $c<1$ uniformly in $\epsilon, t, \eta_{*}(\epsilon, N)$.
\end{itemize}

Given these result, we can establish Theorem \ref{thm:main}.
\begin{proof}[Proof of Theorem \ref{thm:main}]
  By assumption it holds that at time $0$,
  \begin{align*}
    E_1|_{t=0} &\leq \epsilon^2, \\
    E_2|_{t=0} &\leq \epsilon.
  \end{align*}
  In particular, at that time the estimates \eqref{eq:assumpE1} and
  \eqref{eq:assumpE2} are satisfied with improved constants.
  Thus, by continuity these estimates remain true at least for some positive
  time.
  We hence define $0<t_*\leq T$  as the maximal time such that 
  \begin{align*}
    E_1 &\leq 16 \epsilon^2, \\
    E_2 &\leq 4 \epsilon (1+t)^{-\sigma+1},
  \end{align*}
  holds for all times $0\leq t \leq t_*$.
  
  If $t_*=T$, then Theorem \ref{thm:main} immediately follows from the results
  of Propositions \ref{prop:E2} and the discussion following Proposition \ref{prop:E1}.
  Thus suppose for the sake of contradiction, that the maximal time $t_*$ is
  strictly smaller than $T$.
  Then by the estimates of Proposition \ref{prop:E2} at the time $t_{*}$ it
  holds that
  \begin{align*}
    E_2(t) \leq \frac{1}{1-c} \epsilon (1+t)^{-\sigma+1} \leq 8  \epsilon (1+t)^{-\sigma+1},
  \end{align*}
  where we used that $c$ is sufficiently small. Thus equality in
  \eqref{eq:assumpE2} is not attained.
  
  Similarly, by the results of Proposition \ref{prop:E1} and \eqref{eq:IntIneq}
  it follows that 
  \begin{align*}
     E_1 &\leq \epsilon^2 (1+ \int_0^t c (1+s)^{-\sigma+2} ds)   \leq 8 \epsilon^2.
  \end{align*}
  Hence, equality is also not attained at time $t_*$ for \eqref{eq:assumpE1}.
  Thus, by by continuity the estimates \eqref{eq:assumpE2}, \eqref{eq:assumpE1}
  remain valid at least for a small additional time past $t_*$.
  However, this contradicts the maximality of $t_*$ and thus it is not possible
  that $t_*<T$, which concludes the proof.
\end{proof}
It thus remains to establish the bounds for $E_1$ and $E_2$ as claimed in
Propositions \ref{prop:E1} and \ref{prop:E2}.
We emphasize that (except for the cut-off) the estimate for $E_1$ follows the
same abstract and rather rough argument as in \cite{grenier2020landau}, which
requires $\sigma>3$.
In contrast the proof of the estimates of $E_2$ could in principle be modified
to allow $\sigma\geq 1$ and to match the bounds of the echo chains discussed in Section~\ref{sec:model}.
We expect that using methods closer to the ones of \cite{bedrossian2016landau}
it should be possible to obtain the case $\sigma=1$ also for $E_2$, however the method of
\cite{grenier2020landau} seems to require that $\sigma>3$.

\section{Control of $\rho$}
\label{sec:E2}
In this section we assume that $E_1(z(t))$ remains small as stated in
\eqref{eq:assumpE1}:
\begin{align*}
  E_1(z(t))\leq 16 \epsilon^2,
\end{align*}
and estimate the possible norm growth of $\rho(t)$.

We may thus consider $g(t,x,v)$ as given in \eqref{eq:VPsys} and study the
evolution equation for $\rho$:
\begin{align*}
  \dt \rho &=- \int F(t,x) (\nabla_v + t \nabla_x)g(t,x-tv,v) dv.
\end{align*}
Here we note that $F=F[\rho]$ is given by a Fourier multiplier
\begin{align*}
  \mathcal{F} F (t,l) = \frac{1}{il} \tilde{\rho}(t,l)
\end{align*}
and that multiplication turns into a (discrete) convolution under a Fourier transform.
We may thus equivalently express our equation in integral form as 
\begin{align}
  \label{eq:intrho}
  \tilde{\rho}(t,k) = \tilde{g}(0,k,kt) + \sum_{l\neq 0}\int_0^t \frac{k(t-s)}{l}\tilde{\rho}(s,l) \tilde{g}(s,k-l,kt-ls) ds.
\end{align}
As discussed in the echo model of Section \ref{sec:model}, here our assumed
control of $g$ is rather weak if $k-l$ is small and $kt-ls\approx 0$.
Indeed, at a heuristic level we may only expect estimates of the form
\begin{align*}
  |\tilde{g}(s,k-l,kt-ls)| \lesssim \epsilon (1+|kt-ls|^2)^{-1},
\end{align*}
whose time integral is estimated by $\frac{\epsilon}{l}$.
Furthermore, inserting the assumption that $kt-ls\approx 0$ and denoting
$\eta:=kt$, the first factor can be estimated by
\begin{align*}
  \frac{k(t-s)}{l}\approx \frac{(l-k)s}{l} \approx (l-k) \frac{\eta}{l^2}.
\end{align*}
Hence, we again arrive at the estimate of the growth factor by $\epsilon
\frac{\eta}{l^3}$ (for $l-k=1$) as in \eqref{eq:growthfactor}. In particular,
using that $l\approx k = \frac{\eta}{t}$ this factor is much smaller than $1$ if
$\eta$ is very large, which will allow us to introduce a cutoff in $\eta$.

For ease of reference, we also note some general techniques to be used
throughout the proof:
\begin{itemize}
\item The cut-off weight of Definition \ref{defi:gen}
  \begin{align*}
    \cutoff{k,\eta}=
    \begin{cases}
      \epsilon^{\beta} \langle k,\eta \rangle^{1/3}, \\
      \epsilon^{\beta'} \langle k,\eta \rangle^{\gamma}
    \end{cases}
  \end{align*}
  satisfies a triangle inequality
  \begin{align*}
    \cutoff{k,\eta} \leq \cutoff{k-l,\eta-\xi} + \cutoff{l,\xi}. 
  \end{align*}
  Moreover, we emphasize that due to the exponents $0<\gamma\leq 1/3$  the right-hand-side in
  general is much larger than the left-hand-side unless $(l,\xi)$ (or $(k-l,
  \eta-\xi)$) is small compared to $(k,\eta)$.
\item Let $C_1>0$, then for $k, \eta$ it holds that
  \begin{align}
  \label{eq:C13}
  \exp(-C_1 \cutoff{k,\eta}) \leq
    \begin{cases}
      C_{1}^{-3}\epsilon^{-3\beta} \langle k,\eta \rangle^{-1} & \text { below the cut-off},\\
      C_1^{-\frac{1}{\gamma}} \epsilon^{-3\beta'} \langle k,\eta \rangle^{-1} & \text { above the cut-off}.
    \end{cases}
  \end{align}
  We stress that this limits the choice of $\beta$ in Definition
  \ref{defi:gen} to $\beta\leq \frac{1}{3}, \beta'\leq \gamma$.
  Furthermore, it deteriorates for $C_1>0$ small.
\item Similarly, in view of the time decay with $t^{-\sigma+1}$ as required in the estimate
  \eqref{eq:assumpE2}, for some estimates we require an improved bound of the form
  \begin{align}
    \label{eq:Cb}
     \exp(-C_1 \cutoff{k,\eta}) \leq
    \begin{cases}
      C_{1}^{-3\sigma}\epsilon^{-3\sigma\beta} \langle k,\eta \rangle^{-\sigma} & \text { below the cut-off},\\
      C_1^{-\frac{\sigma}{\gamma}} \epsilon^{-3\sigma \beta'} \langle k,\eta \rangle^{-\sigma} & \text { above the cut-off}.
    \end{cases}
  \end{align}
uniformly in $t$. This imposes the stronger condition $\beta\leq
\frac{1}{3\sigma}$.
\end{itemize}

With these preparations, we turn to our estimate of $E_2$.
\begin{proof}[Proof of Proposition \ref{prop:E2}]
Recalling the Definition \ref{defi:gen} of $E_2$, we consider a weighted version
of the integral equation \eqref{eq:intrho}:
\begin{align*}
  & \quad e^{Cz(t) \epsilon^{\beta}\cutoff{k,kt}} \langle k,kt\rangle^{\sigma} |k|^{-\alpha}\tilde{\rho}(t,k) \\
  &= e^{Cz(t) \epsilon^{\beta}\cutoff{k,kt}} \langle k,kt\rangle^{\sigma} |k|^{-\alpha}\tilde{g}(0,k,\eta+kt) \\
  & \quad +\sum_{l\neq 0}\int_0^t e^{Cz(t) \epsilon^{\beta}\cutoff{k,kt}} \langle k,kt\rangle^{\sigma} |k|^{-\alpha} \frac{k(t-s)}{l}\tilde{\rho}(s,l) \tilde{g}(s,k-l,kt-ls) ds.
\end{align*}
The first summand here is rapidly decaying due to the regularity of the initial data, in
particular, it holds that
\begin{align*}
  \sup_{k}  e^{Cz(t) \epsilon^{\beta}\cutoff{k,kt}} \langle k,kt\rangle ^{\sigma} |k|^{-\alpha}\tilde{g}(0,k,\eta+kt) \leq \epsilon (1+t)^{-\sigma}.
\end{align*}

For the integral term we insert the weights of Definition \ref{defi:gen} for both
$\tilde{\rho}$ and $\tilde{g}$ thus arrive at
\begin{align*}
  &\sum_{l\neq 0}\int_0^t e^{C\epsilon^{\beta} (z(t) \cutoff{k,kt}^{1/3}- z(s) \cutoff{k-l,kt-ls}^{1/3} -z(s)\cutoff{l,ls}^{1/3}}\\
  & \quad \langle k,kt\rangle^{\sigma}\langle k-l,kt-ls\rangle^{-\sigma} \langle l,ls\rangle ^{-\sigma}   |k|^{-\alpha}|l|^{\alpha} \frac{k(t-s)}{l}\\
&\left(A(s,l,ls)|l|^{-\alpha} \tilde{\rho}(s,l)  \right) \left(A(s,k-l,kt-ls) \tilde{g}(s,k-l,kt-ls)  \right) ds,
\end{align*}
where we denote the weights as $A$ for brevity.
Using the Sobolev embedding $H^1\subset L^\infty$ and \eqref{eq:assumpE1} to estimate
\begin{align*}
  \|\left(A(s,k-l,kt-ls) \tilde{g}(s,k-l,kt-ls)  \right)\|_{\ell^\infty} \leq \sqrt{E_1(z(s))} \lesssim \epsilon
\end{align*}
and that by definition 
\begin{align*}
  \|\left(A(s,l,ls)|l|^{-\alpha} \tilde{\rho}(s,l)  \right)\|_{\ell^\infty} = E_1(z(s)),
\end{align*}
we thus arrive at an integral inequality of the form
\begin{align*}
  E_1(z(t)) \leq \epsilon (1+t)^{-\sigma} + \sup_k \sum_{l\neq 0}\int_0^t  \epsilon C_{k,l}(t,s) E_1(z(s)) ds.
\end{align*}
Here we use a similar notation as \cite{grenier2020landau} and defined
\begin{align}
  \label{eq:Ckl}
  \begin{split}
  \epsilon C_{k,l}(t,s)&=  e^{C (z(t) -z(s)) \cutoff{k,kt}} e^{C z(s)(\cutoff{k,kt}-\cutoff{k-l,kt-ls} -\cutoff{l,ls})} \\
                       & \quad \langle k,kt\rangle^{\sigma}\langle k-l,kt-ls\rangle ^{-\sigma} \langle l,ls\rangle^{-\sigma} |k|^{-\alpha}|l|^{\alpha}\\
  & \quad \epsilon \frac{k(t-s)}{l}.
\end{split}
\end{align}
In view of the time decay encoded in 
\eqref{eq:assumpE1}, we define (with the same notation as
\cite{grenier2020landau})
\begin{align*}
  \zeta(t)&= \sup_{0\leq \tau \leq t} E_1(z(\tau))\langle s \rangle^{\sigma-1}, \\
  \langle s \rangle &:= \sqrt{1+s^2}.
\end{align*}
Inserting this definition, we obtain the following integral inequality:
\begin{align}
  \zeta(t) \leq \epsilon  + \sup_{k}\int_0^t\sum_{l\neq 0} \epsilon C_{k,l}(t,s)\langle s \rangle^{-\sigma+1} \langle t \rangle^{\sigma-1} ds \zeta(t).
\end{align}
It thus suffices to show that
\begin{align}
  \label{eq:E2est}
  \sup_{k}\int_0^t\sum_{l\neq 0} \epsilon C_{k,l}(t,s)\langle s \rangle^{-\sigma+1} \langle t \rangle^{\sigma-1} ds \leq \frac{1}{2}.
\end{align}
Compared to \cite{grenier2020landau} we here additionally make use of the power $\epsilon^1$.
We further highlight the strong control required for small times $s\ll t$ when
$\sigma>1$. In that region we need to rely on the exponential factors in
$C_{k,l}(t,s)$ to provide decay, which we noted as \eqref{eq:Cb}.

In the following we distinguish multiple cases for the estimates in
\eqref{eq:E2est} and possibly restrict to sub-intervals of $(0,t)$ to study
\begin{align}
  \label{eq:E2integral}
  \int_I \epsilon C_{k,l}(t,s)\langle s \rangle^{-\sigma+1} \langle t \rangle^{\sigma-1} ds
\end{align}
for fixed $k$ and $l\neq 0$.

As a first instructive case we study the setting $l=k$, where we use different
arguments if $|kt-ls|\geq \frac{kt}{2}$ and when $|kt-ls|\leq\frac{kt}{2}$.

\underline{The diagonal case $l=k$:}\\
We note that in this special case $l=k$ (recall that $l\neq 0$) our estimate
\eqref{eq:E2integral} reduces to
\begin{align*}
  \int_I \epsilon &e^{C (z(t) -z(s)) \cutoff{k,kt}} e^{C z(s)(\cutoff{k,kt}-\cutoff{0,k(t-s)} -\cutoff{k,ks})} \\
                       & \quad \langle k,kt\rangle^{\sigma}\langle 0,k(t-s)\rangle ^{-\sigma} \langle k,ks\rangle^{-\sigma} (t-s) \langle t \rangle^{\sigma-1}\langle s \rangle^{-\sigma+1} ds.
\end{align*}
We split this integral into the regions
\begin{align*}
  |kt-ks|\geq \frac{kt}{2} &\Leftrightarrow s \leq \frac{t}{2}, \\
  |kt-ks|\leq \frac{kt}{2} &\Leftrightarrow s \geq \frac{t}{2}.
\end{align*}
Both regimes exhibit similar behavior as the model problems of Section
\ref{sec:model} and hence require different arguments.
\\

We begin our discussion with the region $s\geq\frac{t}{2}$, where it holds that
\begin{align*}
  \langle k,kt\rangle^{\sigma}\langle 0,k(t-s)\rangle ^{-\sigma} \langle k,ks\rangle^{-\sigma} (t-s) \langle t \rangle^{\sigma-1}\langle s \rangle^{-\sigma+1} \leq 2^{2\sigma-1} \langle t-s\rangle^{-\sigma+1} k^{-\sigma}.
\end{align*}
If $\sigma>2$ this integral is bounded uniformly in $t$ and smaller than a
constant times $\frac{\epsilon}{k^\sigma}$.
We emphasize that in this region we did not require any bounds involving $\cutoff{k,kt}$.
\\

We next turn to the case $s\leq \frac{t}{2}$, where we emphasize that 
\begin{align*}
  & \quad \epsilon \langle k,kt\rangle^{\sigma}\langle 0,k(t-s)\rangle ^{-\sigma} \langle k,ks\rangle^{-\sigma} (t-s) \langle t \rangle^{\sigma-1}\langle s \rangle^{-\sigma+1} \\
  &\leq \epsilon \langle t \rangle^{\sigma} \langle k,ks\rangle^{-\sigma}\langle s \rangle^{-\sigma+1}
\end{align*}
is integrable but the value of the integral might be of size 
\begin{align}
  \label{eq:kkcase}
  \epsilon \langle t \rangle^{\sigma} \langle k \rangle^{-\sigma}.
\end{align}
Since $t\leq T$, this bound becomes small if $k: |k|\gg T$ is sufficiently
large, which hence allows for a cutoff at $\eta_*\geq T^2$.
However, for smaller values of $k$ this integral might be very large and we thus
need to rely on our exponential factor to improve our estimate.

For this purpose we note that for $s\leq \frac{t}{2}$, by the intermediate
value theorem and monotonicity of $\dt z$ it holds that
\begin{align*}
  |z(t)-z(s)| = |\p_tz(\overline{t})| (t-s) \geq \frac{t-s}{t} \geq \frac{1}{2}.
\end{align*}
We may thus use the exponential estimate \eqref{eq:Cb} to conclude that
\begin{align*}
  e^{C (z(t) -z(s)) \cutoff{k,kt}} \lesssim
  \begin{cases}
    C^{-3\sigma} \epsilon^{-3\sigma\beta} \langle k,kt \rangle^{-\sigma},& \text{ below the cutoff}, \\
    C^{-\sigma/\gamma} \epsilon^{-\sigma \beta'/\gamma} \langle k,kt \rangle^{-\sigma}, & \text{ above the cutoff}.
  \end{cases}
\end{align*}
Combining this with \eqref{eq:kkcase} and the requirements \eqref{eq:betaconstraints}, we hence obtain a bound uniformly in $t$
and $\epsilon$, which further decays in $k$.
\\

In the following we discuss the various cases $0\neq l \neq k$, where in view of
Section \ref{sec:model} we expect to encounter contributions due to resonances.
Furthermore, as seen already in the simple example of the diagonal case, the
weight $ \langle t \rangle^{\sigma-1}\langle s \rangle^{-\sigma+1}$ in
\eqref{eq:E2est} poses challenges to deriving good estimates in terms of powers
of $\epsilon$.\\

\underline{The reaction case: $\left\{s: |kt-ls|\leq \frac{kt}{2}\right\}$.}\\
We note that in this case by the triangle inequality it holds that $|ls|\geq \frac{kt}{2}$ and
thus \eqref{eq:E2integral} reduces to
\begin{align*}
  \int_{I} e^{C(z(t)-z(s))\cutoff{k,kt}}  l \langle k-l,kt-ls \rangle^{-\sigma} |k|^{-\alpha}|l|^{\alpha} \epsilon \frac{k(t-s)}{l^2} \langle t \rangle^{\sigma-1} \langle s \rangle^{-\sigma+1} ds.
\end{align*}
We first consider the sub-case $s\leq \frac{t}{2}$.
Here we note that
\begin{align*}
  |z(t)-z(s)| &\geq \frac{1}{2}, \\
  t-s &\leq t, \\
  |l| &\geq k.
\end{align*}
Therefore, a rough estimate is given by
\begin{align*}
    e^{-\frac{C}{2}\cutoff{k,kt}} \epsilon \langle t \rangle^{\sigma} \int_{I} l \langle k-l,kt-ls \rangle^{-\sigma}  \langle s \rangle^{-\sigma+1} ds.
\end{align*}
The integral here is bounded by $\langle k-l \rangle^{-\sigma+1}$, which is
summable in $l$ provided $\sigma>2$.
For the prefactor, we use \eqref{eq:Cb} and \eqref{eq:betaconstraints} to
control
\begin{align*}
   e^{-\frac{C}{2}\cutoff{k,kt}} \epsilon \langle t \rangle^{\sigma} \lesssim C^{-3\sigma}.
\end{align*}
Therefore, we indeed obtain a bound by a small constant provided $C$ is
sufficiently large.

We next consider the sub-case $\frac{t}{2}\leq s \leq t$.
Here we note that
\begin{align*}
  |z(t)-z(s)|&\geq \frac{t-s}{t}, \\
  \langle t \rangle^{\sigma-1} \langle s \rangle^{-\sigma+1} &\leq 2^{\sigma-1}, \\
  \frac{k}{2}\leq l &\leq 2k.
\end{align*}
Therefore, may equivalently consider
\begin{align*}
  \int_{I} e^{-C \frac{t-s}{t} \cutoff{k,kt}}  \langle k-l, kt-ls \rangle^{-\sigma} \epsilon (t-s) |k|^{1-\alpha} |l|^{\alpha-1}.
\end{align*}
If $|kt-ls|\geq \frac{t}{2}\geq \frac{t-s}{2}$, we may simply estimate
$|k|^{1-\alpha} |l|^{\alpha-1}\leq 2^{\alpha-1}$ and
\begin{align*}
  \langle k-l, kt-ls \rangle^{-\sigma} \epsilon (t-s) \leq 2 \epsilon  \langle k-l, kt-ls \rangle^{-\sigma+1},
\end{align*}
which is integrable. Furthermore, the value of the integral is bounded by
$\epsilon\langle k-l \rangle^{-\sigma+2}$ and hence summable and small provided
$\sigma>3$.
It thus remains to study the sub-case, where $|kt-ls|\leq \frac{t}{2}$, $s\geq
\frac{t}{2}$, which implies that
\begin{align*}
  |t-s| \approx \frac{|k-l|}{k}
\end{align*}
We thus need to control
\begin{align}
  \label{eq:key13}
  e^{-C \frac{|k-l|}{k} \cutoff{k,kt}} (k-l)  \epsilon \frac{kt}{k^3}  l \langle k-l, kt-ls \rangle^{-\sigma+1}. 
\end{align}
Applying the estimate \eqref{eq:C13} \emph{below the cut-off}, as in \cite[Lemma
4.4]{grenier2020landau} we hence obtain a bound by
\begin{align*}
   \epsilon^{1-3\beta} |k-l|^{-2} k^3 \frac{k-l}{k^3} \langle k-l, kt-ls \rangle^{-\sigma+1}.
\end{align*}
The powers of $k$ exactly cancel and we obtain the desired bound.

We emphasize that \emph{above the cut-off} this argument breaks, since we would
obtain a bound by $k^{\frac{1}{\gamma}-3}$ which grows unbounded as
$k\rightarrow \infty$. However, this problem does not occur in the case of a
finite time interval, since one then \emph{cannot} independently let $kt$ and
$k^3$ tend to infinity.
More precisely, since $0 \leq t \leq T$ we may very roughly control
\eqref{eq:key13} by
\begin{align*}
  \epsilon \frac{kT}{k^3} \langle k-l, kt-ls \rangle^{-\sigma+2},
\end{align*}
\emph{irrespective of the definition of} $\cutoff{k,kt}$!
In particular, if $k$ is sufficiently large such that
\begin{align*}
  \epsilon \frac{kT}{k^3} \ll 1
\end{align*}
no further argument is required.
Since our cut-off is defined in terms of
\begin{align*}
  \langle k,kt \rangle \geq \eta_*,
\end{align*}
this can for instance be achieved by choosing $\eta_*\geq T^2$, so that $k\geq T$.

As discussed in Section \ref{sec:model}, the choice of $\eta_*$ can surely be
further optimized, but for the purposes of this article a rough bound is sufficient.

\underline{The transport case: $\left\{s: \frac{kt}{2}\leq |kt-ls|\right\}$.}\\
Similarly to the second model of Section \ref{sec:model}, in this case $ls$ and $l$ might be very small and hence cannot compensate for
powers of $k$ and $t$. For this reason we crucially rely on the decay of the
exponential factor. In particular, in contrast to the $l=k$ case, we here need
to control positive powers of $k$ even if $kt$ is much larger than the cut-off.

Inserting our assumptions we need to estimate
\begin{align*}
        \int_I \epsilon &e^{-C\epsilon^{\beta} (z(t)-z(s)) \cutoff{k,kt}^{1/3}}  \\
   &\quad 2^\sigma  l  \langle l,ls\rangle ^{-\sigma} |k|^{-\alpha}|l|^{\alpha}\\
  & \quad  \frac{k(t-s)}{l^2}  \langle t \rangle^{\sigma-1} \langle s \rangle^{-\sigma+1} ds.
\end{align*}
We argue as in \cite{grenier2020landau} (see Lemma 4.4, case
1 and, in particular, $(4.24)$ there), but more discussion
is required for the cut-off.

We first discuss the case $s\leq \frac{t}{2}$.
Following a similar argument as in the resonant case we bound $|z(t)-z(s)|\geq
\frac{1}{2}$ below and apply \eqref{eq:Cb} to arrive at bound by
\begin{align}
  \label{eq:mainrestriction}
  \begin{split}
  & \quad k^{1-\alpha} \langle t \rangle^{\sigma} |l|^{\alpha-1} \langle l, ls \rangle^{-\sigma}\langle s \rangle^{-\sigma+1} \langle k,kt\rangle^{-\sigma} C^{-3\sigma} \\
  &\leq C^{-3\sigma} k^{1-\alpha -\sigma}  |l|^{\alpha-1} \langle l, ls \rangle^{-\sigma}\langle s \rangle^{-\sigma+1}.
\end{split}
\end{align}
Here we estimated $C^{-\frac{\sigma}{\gamma}}\leq C^{-3\sigma}$ for simplicity
of notation.

For $s\geq \frac{t}{2}$ also further discussion is needed. If $l\geq \frac{k}{2}$, we
simply bound by
\begin{align*}
  & \quad\int_I \epsilon  \langle l,l t\rangle ^{-\sigma} |k|^{-\alpha}|l|^{\alpha}\frac{kt}{l^2} ds \\
  & \leq \langle t \rangle^{-\sigma+2} l^{-\sigma+\alpha-2} |k|^{-\alpha+1},
\end{align*}
which is summable in $l\geq \frac{k}{2}$ provided $\sigma>2$.

For $|l|\leq \frac{k}{2}$ we instead can only use that $|l|^{\alpha-2}$ is
summable and need to show that
\begin{align*}
  \int_{I} \epsilon &e^{-C\epsilon^{\beta} (z(t)-z(s)) \cutoff{k,kt}^{1/3}} |k|^{1-\alpha}(t-s) \langle t \rangle^{-\sigma} ds
\end{align*}
is uniformly bounded. Following the argument of \cite{grenier2020landau} we
employ a variant of \eqref{eq:Cb} with the exponent $3(1-\alpha)$ when
\emph{below the cut-off} to obtain
\begin{align*}
  |t-s|^{1-3(1-\alpha)} t^{3(1-\alpha)} k^{1-\alpha} \langle k,kt\rangle^{-(1-\alpha)}\langle t \rangle^{-\sigma}.
\end{align*}
At this point we require that
\begin{align*}
  |t-s|^{1-3(1-\alpha)}
\end{align*}
is locally integrable and hence that
\begin{align*}
  1-3(1-\alpha) > -1 \Leftrightarrow \alpha> \frac{1}{3}.
\end{align*}
Integrating and bounding 
\begin{align*}
  \int_{I} |t-s|^{1-3(1-\alpha)} ds \leq \langle t \rangle^{2-3(1-\alpha)},
\end{align*}
We hence arrive at a bound by
\begin{align*}
  \langle t \rangle^{2-\sigma} \leq 1.
\end{align*}
If we are \emph{above the cut-off} we argue similarly and apply \eqref{eq:Cb}
with the exponent $\frac{1-\alpha}{\gamma}$, which imposes the stronger
constraint
\begin{align}
  \label{eq:alphaconstraint}
  1-\frac{1-\alpha}{\gamma} > -1 \Leftrightarrow \alpha > 1-2\gamma.
\end{align}
Thus the transport case should be understood to impose constraints on $\alpha$
given $\gamma$.

\end{proof}

This concludes our estimate of $\rho(t)$ incorporating a frequency cut-off and
improved $\epsilon$ dependence.
The estimate of $g(t,x,v)=f(t,x-tv,v)$ in comparison uses a much more abstract and
lossy argument. In particular, we can follow the strategy of
\cite{grenier2020landau} more closely and only need to discuss the cut-off in
some detail.

\section{Control of $f(t,x-tv,v)$}
\label{sec:E1}
In this section we establish growth bounds on $E_1$ with $\rho(t,x)$ considered
given and, in particular, prove Proposition \ref{prop:E1}.
Here we argue similarly as in Proposition 4.1 of \cite{grenier2020landau}, but
need to account for the following changes:
\begin{itemize}
\item Our cut-off $\cutoff{\cdot,\cdot}$ include a factor $\epsilon^{\beta}$ in the exponent,
  hence compared to \cite{grenier2020landau} our derivative $\p_z$ is rescaled
  by $\epsilon^{-\beta}$.
\item Above the cut-off we use a different exponent $\gamma$, which we need to
  account for in our estimates.
\end{itemize}

\begin{proof}[Proof of Proposition \ref{prop:E1}]
  We follow the same strategy as in \cite[Proposition 4.1]{grenier2020landau}
  and denote our weight as
  \begin{align*}
    A_{k,\eta}:= e^{C z \cutoff{k,\eta}^{1/3}} \langle k,\eta\rangle^\sigma.
  \end{align*}
  We remark that these cut-offs preserve triangle inequalities and that
  $A_{k,\eta}\leq C A_{k-l,\eta-\xi}A_{l,\xi}$ satisfies an algebra property
  (with constant independent of $\eta_*$).

  Since we consider the case of perturbations around $0$, after a Fourier
  transform the Vlasov-Poisson equations read
  \begin{align*}
    \dt \tilde{g}(k,\eta) = -i \sum_{l} (\eta-kt) \tilde{F}(t,l) \tilde{g}(t,k-l, \eta-lt).
  \end{align*}
  Testing with $A_{k,\eta}^2\tilde{g}(k,\eta)$, we note that by anti-symmetry
  \begin{align*}
    \int \sum_{k,l} -i (\eta-kt) \tilde{F}(t,l) A_{k-l,\eta-lt}\tilde{g}(t,k-l,\eta-lt) A_{k,\eta} \tilde{g}(t,k,\eta) d\eta =0
  \end{align*}
  and hence our estimate reduces to controlling
  \begin{align}
    \label{eq:gest}
   \int -i\sum_{k,l} (\eta-kt) \frac{A_{k,\eta}-A_{k-l,\eta-lt}}{A_{l,lt}A_{k-l,\eta-lt}} A\tilde{F}(t,l) A\tilde{g}(t,k-l, \eta-lt)  A\tilde{g}(t,k,\eta) d\eta.
  \end{align}
  in terms of $\p_z E_1$.

  Here we argue as in \cite{grenier2020landau} and first estimate
  \begin{align*}
    (\eta-kt) \frac{A_{k,\eta}-A_{k-l,\eta-lt}}{A_{l,lt}A_{k-l,\eta-lt}}
  \end{align*}
  beginning with the case where
  \begin{align*}
    \langle l,lt \rangle \geq \frac{1}{2} \langle k, \eta \rangle.
  \end{align*}
  Then it follows that
  \begin{align}
    \label{eq:gcase1}
    \begin{split}
    & \quad \langle t \rangle \langle k,\eta \rangle \frac{A_{k,\eta}-A_{k-l,\eta-lt}}{A_{l,lt}A_{k-l,\eta-lt}} \\
    &\leq \langle t \rangle \langle k,\eta \rangle \frac{\langle k,\eta \rangle^{\sigma} + \langle k-l,\eta-lt \rangle^{\sigma}}{\langle l,lt \rangle^{\sigma}\langle k-l,\eta-lt \rangle^{\sigma}}\\
    & \leq \langle t \rangle (\langle k-l \rangle^{-\sigma+1} + \langle l \rangle^{-\sigma+1}).
  \end{split}
  \end{align}
  If instead $\langle l,lt\rangle \leq \frac{1}{2} \langle k,\eta \rangle$ is
  possibly much smaller, we need to exploit cancellation in
  \begin{align*}
  \langle k,\eta \rangle (A_{k,\eta} - A_{k-l, \eta-lt}).
  \end{align*}
  More precisely, we recall that $A_{k,\eta}$ is of the form
  \begin{align*}
    \langle x \rangle^{\sigma} e^{\cutoff{x}}.
  \end{align*}
  We thus split
  \begin{align*}
    \langle x \rangle^{\sigma} e^{\cutoff{x}} - \langle x+y \rangle^{\sigma} e^{\cutoff{x+y}}\\
    =  (\langle x \rangle^{\sigma} - \langle x+y \rangle^{\sigma}) e^{\cutoff{x}}\\
    + \langle x+y \rangle^{\sigma} (e^{\cutoff{x}}- e^{\cutoff{x+y}}),
  \end{align*}
  By the chain rule and the intermediate value theorem
  \begin{align*}
    |\langle x \rangle^{\sigma} - \langle x+y \rangle^{\sigma}| \leq  c_{\sigma}  \frac{\langle y \rangle}{\langle x \rangle} \langle x \rangle^{\sigma},
  \end{align*}
  where we used that $\frac{|x|}{2} \leq |x+y|\leq 2|x|$.
  Hence, for that part we easily obtain a bound by
  \begin{align*}
    \langle l,lt \rangle (A_{k,\eta} - A_{k-l, \eta-lt}).
  \end{align*}
  It thus remains to discuss the difference
  \begin{align*}
    e^{\cutoff{x}}- e^{\cutoff{x+y}}.
  \end{align*}
  Since $x$ and $x+y$ are of comparable magnitude and we consider a smooth
  cut-off (that is, matching powers $\beta$, $\beta'$), it suffices to consider
  the case when either both $x$ and $x+y$ are above or below the cut-off.
  In either case this hence reduces to considering
  \begin{align*}
    & \quad  e^{C_1 \langle x \rangle^{\theta}} - e^{C_1 \langle x+y \rangle^{\theta}} = \int_0^1 \frac{d}{d\tau} e^{C_1 \langle x+ \tau y \rangle^{\theta}} d\tau\\
    &\leq \int_0^1 C_1 \theta \langle y \rangle \langle x+ \tau y \rangle^{\theta-1} e^{C_1 \langle x+ \tau y \rangle^{\sigma}} d\tau \\
    & \leq 2 C_1 \langle y \rangle \langle x \rangle^{\theta-1} (e^{C_1 \langle x \rangle^{\theta}} + e^{C_1 \langle x+y \rangle^{\theta}}) \\
    & \leq 2 \frac{\langle y \rangle}{\langle x \rangle}  (C_1\langle x \rangle^{\theta} e^{C_1 \langle x \rangle^{\theta}}  +C_1 \langle x+y \rangle^{\theta} e^{C_1 \langle x+y \rangle^{\theta}} ).
  \end{align*}
  We remark that here by our notational conventions
  \begin{align*}
    C_1\langle x \rangle^{\theta} e^{C_1 \langle x \rangle^{\theta}} = C^{-1} \epsilon^{-\beta} \p_z e^{Cz \cutoff{x}}.
  \end{align*}

  We hence obtain the desired control and may conclude as in
  \cite{grenier2020landau}. In the interest of a self-contained presentation, we
  recall the proof below:
  Inserting the previous estimates \eqref{eq:gest} may be controlled by
  \begin{align*}
  \epsilon^{-\beta} \sum_{k,l} \langle t \rangle (\langle k-l \rangle^{-\sigma+1} + \langle l \rangle^{-\sigma+1})\\
    \int  (AF)_{l,lt} \sqrt{1+\cutoff{k,\eta}}(Ag)_{k,\eta} \sqrt{1+\cutoff{k-l,\eta-lt}}(Ag)_{k-l,\eta-lt}.
  \end{align*}
  Here we now make use of the power $|l|^{-\alpha}$ in the definition of $E_2$
  and the fact that $F_{l}=\frac{1}{il}\rho_{l}$
  and bound
  \begin{align*}
    |(AF)_{l,lt}| \leq \langle l \rangle^{\alpha-1} E_2.
  \end{align*}
  where
  \begin{align*}
    \langle l \rangle^{\alpha-1}  \in \ell^2(\Z) \Leftrightarrow \alpha<\frac{1}{2}. 
  \end{align*}
  Thus, by Young's and Hölder's inequality we over all obtain a control by
  \begin{align*}
    c \epsilon^{-\beta} E_{2} (1+\p_z)E_1. 
  \end{align*}
  The estimate for $\p_{\eta}g_{k\eta}$ follows analogously and is hence
  omitted (see also \cite[Proposition 4.1]{grenier2020landau}). 
\end{proof}

To the author's knowledge this is the first nonlinear Landau damping result (for
finite times $\epsilon^{-N}$) in Gevrey classes larger than $3$ (see also Lemma
\ref{lem:trivial} and \cite{grenier2022plasma}).
As mentioned throughout the proofs, we expect that optimal classes are
determined by plasma echoes as captured in the first model problem of Section \ref{sec:model}. However, the method
of \cite{grenier2020landau} which we adapt trades further decay for
simplicity of proof. That is, the additional requirement that 
\begin{align*}
  E_2 \leq \epsilon (1+t)^{-\sigma+1}
\end{align*}
with $\sigma>3$ further restricts our choice of parameters and seems to be
optimal for this method of proof.

While we for simplicity here have considered the case of Landau damping around the
special case $f=0$, we expect that these results can be extended to Penrose
stable equilibria by similar arguments as in \cite{grenier2020landau}. Moreover,
in future work we plan to establish nonlinear stability of the wave-type
solutions underlying the echo chain construction of
\cite{bedrossian2016nonlinear} and \cite{zillinger2020landau} in suitable Gevrey classes.

\subsection*{Acknowledgments}
Funded by the Deutsche Forschungsgemeinschaft (DFG, German Research Foundation) – Project-ID 258734477 – SFB 1173.

\bibliography{citations2}
\bibliographystyle{alpha}

\end{document}